\documentclass[12pt]{amsart}

\usepackage{amsfonts}
\usepackage{enumerate}

\date{}

\usepackage{cite}
\usepackage{graphicx}
\usepackage{color}
\usepackage[colorlinks]{hyperref}
\usepackage{amssymb}
\usepackage[utf8]{inputenc}
\usepackage{pstricks-add}
\usepackage{pgf,tikz}
\usetikzlibrary{arrows}

\newtheorem{theorem}{ Theorem}[section]
\newtheorem{corollary}[theorem]{ Corollary}
\newtheorem{lemma}[theorem]{ Lemma}

\begin{document}

\title{\bf Laplacian Spectral Determination of Path-Friendship Graphs}

\author{\bf A.Z. Abdian, A.R. Ashrafi$^\star$, L.W. Beineke and M.R. Oboudi}

\address{\textbf{Ali Zeydi Abdian}, Department of Mathematical Sciences, Lorestan University, College of Science, Lorestan, Khoramabad, Iran, E-mail: azeydiabdi@gmail.com; aabdian67@gmail.com; abdian.al@fs.lu.ac.ir}

\address{\textbf{ Ali Reza Ashrafi},  Department of Pure Mathematics, Faculty of Mathematical Sciences, University of Kashan, Kashan 87317-53153, E-mail: ashrafi@kashanu.ac.ir}

\address{\textbf{Lowell W. Beineke}, Department of Mathematical Sciences, Purdue University Fort Wayne, Fort Wayne, Indiana 46805, USA, E-mail: beineke@pfw.edu}

\address{\textbf{Mohammad Reza Oboudi}, Department of Mathematics, College of Sciences, Shiraz University, Shiraz, 71457-44776, Iran, E-mail: mr\_oboudi@shirazu.ac.ir}

\thanks{$^\star$Corresponding author (Email: ashrafi@kashanu.ac.ir)}

\maketitle
\begin{abstract}
A graph $G$ is said to be determined by the spectrum of its Laplacian matrix (DLS) if every graph with the
same spectrum is isomorphic to $G$. van Dam and Haemers (2003) conjectured that almost all  graphs have this property, but that is known to be the case only for a very few families. In some recent papers it is proved that the friendship graphs and starlike trees are DLS. If a friendship graph and a starlike tree are joined by merging their vertices of degree greater than 2, then the resulting graph is called a path-friendship graph. In this paper, it is proved that the path-friendship graphs are also DLS.

\vskip 3mm

\noindent\textbf{Keywords:} Path-friendship graph, Laplacian matrix, Laplacian spectrum, $L$-cospectral, DLS.

\vskip 3mm

\noindent\textit{2010 Mathematics Subject Classification:} 05C50.
\end{abstract}

\section{Basic Definitions}
Our notation and terminology will follow \cite{CRS}. Let $G = (V, E)$ be a simple graph having $n$ vertices and $m$ edges, with $V = \{v_1, v_2, \ldots, v_n\}$ and $E =\{e_1, e_2, \ldots, e_m\}$. The complement of $G$ is denoted by $\overline{G}$. Several other operations on graphs will be useful here, starting with the disjoint union of $r$ copies of graph $G$ being denoted by $rG$. Consistent with this notation, we let $G + H$ denote the disjoint union of graphs $G$ and $H$. The join $G * H$ of graphs $G$ and $H$ is obtained from $G + H$ by joining each vertex of $G$ to each vertex of $H$. Our next operation applies only to rooted graphs, that is, graphs in which one vertex is singled out as being the root: if $G$ and $H$ are rooted graphs, then their coalescence $G \bullet H$ is obtained from $G + H$ by identifying their roots.

Spectral graph theory originated with the eigenvalues of the adjacency matrix $A(G)$ of a graph $G$, but a second matrix has comparable importance. The \textit{Laplacian matrix} $L(G) = A(G) - D(G)$, where $D(G)$ is the diagonal matrix ${\rm{Diag}}(d_1, d_2, \ldots, d_n)$ in which $d_i$ is the degree of the vertex $v_i$. Let $ \mu_1 $, $\mu_2 $, $\cdots$, $\mu_t$ be the distinct eigenvalues of $L(G)$ with multiplicities $m_1$, $ m_2$, $\cdots$, $m_t$, respectively. The \textit{Laplacian spectrum} or \textit{$L$-spectrum} of $G$ is the multi-set of eigenvalues of $L(G)$ usually written in non-increasing order $\mu_1 \geq \mu_2 \geq \ldots \geq \mu_n=0$.

In recent decades, graphs that are determined by their spectrum have received increasing attention, particularly since they have been applied to a variety of fields, including randomized algorithms, combinatorial optimization, and machine learning. An important part of spectral graph theory is devoted to determine whether a given graph or also a class of graphs is determined by its spectrum or not. So, finding  a new class of graphs that are determined by their spectra can be an interesting and important problem. This is the main focus of this paper.

van Dam and Haemers \cite{VH} conjectured that almost all graphs are determined by their Laplacian spectrum, that is, they are the only graph (up to isomorphism) with that spectrum. However, very few graphs are known to have that property, and so discovering new classes of such graphs is an interesting problem.  Formally, we define two graphs $G$ and $H$ to be \textit{$L$-cospectral} if they have the same $L$-spectrum, and a graph $G$ is \textit{determined by its Laplacian spectrum}, abbreviated by DLS, if no other graphs are $L$-cospectral with $G$.

A wind-wheel graph $G_{s,t}$ on $2s+t+1$ vertices is the graph obtained by appending $s$ triangle(s) to a pendent vertex of path $P_{t+1}$. The lollipop graph of order $n$, denoted by $H_{n,p}$, is obtained by appending a cycle $C_p$ to a pendant vertex of a path $P_{n-p}$. Clearly, the wind-wheel graph and lollipop graph can be obtained respectively by the coalescence operation. With the best of our knowledge, most known DLS-graphs are characterized separately. In general, the DLS-property is no longer preserved under the graph operation. However the coalescence operation of some DLS-graphs will produce new DLS-graphs. In fact, to consider whether the coalescence of some DLS-graphs is also DLS seems an interesting problem.

A friendship graph is a collection of triangles all sharing precisely one vertex and a starlike tree  is a  tree with exactly one vertex of degree greater than $2$. If a friendship graph and a starlike tree are joined by merging their vertices of degree greater than 2, then the resulting graph is called a path-friendship graph.
Abdian et al. \cite{AB} proved that the friendship graphs are DLS (see also \cite{LZG}) and Omidi and Tajbakhsh \cite{OT} and independently Feng and Yu \cite{FY} proved that the starlike trees are DLS. Thus, it is natural to ask about this property for path-friendship graphs. The aim of this paper is to prove that all path-friendship graphs are DLS.

\section{Background Materials}
In this section, some known results are given which are crucial throughout this paper. We also review the most important results on DLS-graphs. Let us start by the main properties of these graphs.

\begin{theorem} [\cite{M, VH, WH}]\label{thm 2-1}
The following can be obtained from the Laplacian spectrum of a graph:
\begin{enumerate}[$i)$]
\item the number of vertices,
\item the number of edges,
\item the number of spanning trees,
\item the number of components,
\item the sum of the squares of the degrees of the vertices.
\end{enumerate}
\end{theorem}

We note that the spectrum of the adjacency matrix of a graph gives other information including the number of closed walks of any given length, whether the graph is bipartite or not, whether it is regular or not, and if it is, the degree of regularity. The next theorem relates the Laplacian spectra of complementary graphs.

\begin{theorem}[\cite{K1,K2}]\label{thm 2-2} Let $\mu_1 \geq \mu_2 \geq \ldots \geq \mu_n=0$ and $\overline{\mu} _1\geq \overline{\mu}_2\geq \ldots \geq \overline{\mu}_n=0$ be the Laplacian spectra of $G$ and $\overline{G}$, respectively. Then $\overline{\mu}_i = n - \mu_{n-i}$ for $i = 1,2,\ldots, n - 1$.
\end{theorem}

For graphs $G$ and $H$, we let $N_G(H)$ be the number of subgraphs of a graph $G$ that are isomorphic to $H$. Further, let $W_G(i)$ be the number of closed walks of length $i$ in $G$ and $W'_H(i)$ be the number of closed walks of length $i$ in $H$ that cover the edges of $H$. Then $W_G(i)=\sum{{N_G(H)W'_H(i)}}$, where the sum is taken over all connected subgraphs $H$ of $G$ for which $W'_H(i)\neq 0$. This equation provides  formulas for calculating the number of some short closed walks in a graph. Note that if tr$(M)$ denotes the trace of the matrix $M$, then $W_G(3) = \mathrm{tr}(A^3(G))$ (with an $n$-cycle having $2n$ closed walks of length $n$).

\begin{theorem}[\cite{O1}]\label{thm 2-3} Suppose $G$ is a graph with exactly $m$ edges. The number of closed walks of lengths $2$, $3$, and $4$ in  $G$ can be computed by  the following formulas:
\begin{enumerate}[$(i)$]
\item $W_G(2)=2m$,
\item $W_G(3) = {\mathrm{tr}(A^3(G))} = 6N_G(C_3)$,
\item $W_G(4)=2m+4N_G(P_3)+8N_G(C_4)$.
\end{enumerate}
\end{theorem}

Turning to the degrees of the vertices in graphs, as before, we let $d_i$ denote the degree of vertex $v_i$ in a graph $G$, and assume that $d_1 \geq d_2 \geq \ldots \geq d_n$. In addition, the eigenvalues of $G$ are $\mu_1 \geq \mu_2 \geq \ldots \geq \mu_n=0$.

\begin{theorem}[\cite{GM}]\label{thm 2-4} If $G$ is a graph with at least one edge, then $\mu_1 \geq d_1 + 1$. Moreover, if $G$ is connected, then equality holds if and only if $d_1 = n-1$.
\end{theorem}

The next result uses the quantity $\theta(v) = \Sigma \dfrac{{\rm{deg}} u}{\rm{deg} v}$, where the sum is taken over the neighbors $u$ of the vertex $v$.

\begin{theorem}[\cite{LP, M}]\label{thm 2-5}
If $G$ is a connected graph, then $\mu_1(G) \leq \max_v(\deg(v)+ \theta(v))$. Moreover, equality holds if and
only if $G$ is a regular  or a semi-regular bipartite graph.
\end{theorem}

\begin{theorem}[\cite{A, LP}]\label{thm 2-6}
If $G$ is a nontrivial graph, then $\mu_1(G) \leq d_1 + d_2$; and if $G$ is connected, then $\mu_2(G) \geq d_2(G)$.
\end{theorem}

\begin{theorem}[\cite{OAJ}]\label{thm 2-7} The first four coefficients in the characteristic polynomial $\varphi(G) = \Sigma l_ix^i$ are $l_0=1$,  $l_1=-2m$, $l_2=2m^2-m-\frac{1}{2}\sum\limits_{i = 1}^{{n}} {{d^2_i}},$
and $l_3$ $=$ $\frac{1}{3}(-4m^3+6m^2+3m\sum_{i = 1}^{{n}} {{d^2_i}}-\sum_{i = 1}^{{n}} {{d^3_i}}-3\sum_{i = 1}^{{n}} {{d^2_i}}+6N_G(C_3)).$
\end{theorem}

The following result is an immediate consequence of Theorem \ref{thm 2-7}.

\begin{corollary}\label{lem 2-10}
\textit{If $G$ and $H$ are $L$-cospectral graphs with the same degree sequences, then they have the same number of triangles.}
\end{corollary}

It follows from Theorems \ref{thm 2-1} and \ref{thm 2-7} that if $G$ and $G'$ are $L$-cospectral graphs with degrees $d_1, d_2, \ldots, d_n$ and $d'_1, d'_2, \ldots, d'_n$ respectively, then $$\mathrm{tr}(A^3(G))-\sum_{i = 1}^{{n}} {{d^3_i}}=\mathrm{tr}(A^3(G'))-\sum_{i = 1}^{{n}} {{d^{'3}_i}}.$$ Based on this equality, Liu and Huang \cite{LH} defined the following graph invariant for a graph $G$:
$$ \varepsilon(G)= \mathrm{tr}(A^3(G))-\sum\limits_{i = 1}^{{n}} {{(d_i - 2)^3}}.$$
\begin{theorem}[\cite{LH}]\label{thm 2-8}
If $G$ and $H$ are $L$-cospectral, then $\varepsilon(G)=\varepsilon(H)$.
\end{theorem}

\begin{theorem}[\cite{O2}]\label{thm 2-9}
The only connected graphs whose largest Laplacian eigenvalue is less than $4$ are paths and odd cycles.
\end{theorem}

Suppose $G(a,b,c,d)$ is a graph with $n = 2a + b + 2c + 3d + 1$ vertices consists of $a$ triangle(s), $b$ pendant edge(s), $c$ pendant path(s) of length 2 and $d$ pendant path(s) of length 3, sharing a common vertex. Ma and Wei \cite{MF} proved that the graph $G(a,b,c,d)$ is determined by its Laplacian spectrum.  Omidi \cite{O2} characterized all graphs with the largest Laplacian eigenvalue at most 4.  As a consequence, he proved that the graphs with the largest Laplacian eigenvalue less than $4$ can be determined by their Laplacian spectra.

There are some other graphs that can be characterized by their Laplacian spectra. These are the friendship graph $F_s$ and butterfly graph $B_{r,s}$ \cite{D, LZG, Wa}, $W_n$ and $S(n; c; k)$ \cite{L, LZG}, $K^m_n$ and $U_{n, p}$ \cite{Z} and firefly graph $F_{s,t, n-2s-2t-1}$ \cite{LGS}.

We conclude this section with a major result known as Cauchy's interlacing theorem \cite{CRS}. It does not explicitly involve graphs, but will prove to be very useful when applied to graphs.

\begin{theorem}[\cite{CRS}]\label{thm 2-10} If $\mu_1\geq \mu_2\geq \ldots \geq \mu_n$ are the eigenvalues of a symmetric $n\times n$ matrix $M$, and if $\lambda_1\geq \lambda_2\geq \ldots \geq \lambda_{n-1}$ are the eigenvalues of a principal submatrix of $M$, then $\mu_1 \geq \lambda_1 \geq \mu_2 \geq \ldots \geq \mu_{n-1} \geq \lambda_{n-1} \geq \mu_n$.
\end{theorem}

\section{Main Results}
In this section, it is proved that  all path-friendship graphs are DLS . Recall that these graphs are defined as the coalescence of a friendship graph rooted at its central vertex and a collection of paths rooted at one end. We note that a starlike tree, often defined as a tree with exactly one vertex of degree greater than 2, can also be thought of as the coalescence of at least three paths rooted at an end, which becomes its root. For convenience, we include a path rooted at any vertex as being a rooted starlike tree, with $1$ or $2$ paths. With this convention, a path-friendship graph $G$ is the coalescence of a friendship graph $F_s$ with $s$ triangles and a starlike tree $T$, that is, $G = F_s \bullet T$. We let $G(s, k)$ denote any such graph in which the friendship graph $F_s$ has $s$ triangles and the starlike tree has $k$ paths. Note that in general it is not uniquely defined.

Our main result builds on the fact that  two constituent parts of friendship graphs are, on their own, each DLS families. For friendship graphs, this was shown by Liu et al. \cite{LZG}, and for starlike trees independently by Feng and Yu \cite{FY} and by Omidi and Tajbakhsh \cite{OT}.

\begin{theorem}[\cite{LZG}]\label{thm 3-1}
All friendship graphs are DLS.
 \end{theorem}

 \begin{theorem}[\cite{FY, OT}]\label{thm 3-2}
All starlike trees are DLS.
 \end{theorem}

The following two lemmas provide some information about the eigenvalues of graphs that are $L$-cospectral with $G(s, k)$.

\begin{lemma}\label{lem 3-1}
If $H$ is a graph that is $L$-cospectral with $G = G(s, k)$, then $$\begin{cases}
\mu_1(H) = 2s+1 & \text{if}\ k = 0\\
2s+k+1\leq \mu_1(H)\leq 2s+k+2 & \text{if}\ k \geq 1. \end{cases}$$
\end{lemma}

\begin{proof} If $k=0$, then $	G=G(s, 0)=F_s$ and the proof is straightforward.
 If $k\geq 1$, then by Lemma \ref{thm 2-4} $\mu_1(G)\geq 2s+k+1$ and by Lemma \ref{thm 2-5} $\mu_1(G)\leq 2s+k+2$. This implies that $2s+k+1\leq \mu_1(H)=\mu_1(G)\leq2s+k+2$ and the proof is completed.
\end{proof}

\begin{lemma}\label{lem 3-2}
If $H$ is a graph that is $L$-cospectral with $G = G(s, k)$, then $\mu_2(H)< 4$.
\end{lemma}

\begin{proof} Let $v$ be a vertex with maximum degree in $G$, and let $M_v$ be the $ (n-1)\times(n-1)$ principal submatrix of $L(G)$ formed by deleting the row and column corresponding to $v$. Since $M_{v}$ contains negative entries, we consider $|M_{v}|$ which is obtained by taking the absolute value of the entries of $M_v$. Now $M_v$ is reducible, but it has $s+k$ irreducible submatrices that correspond to the components of $G-v$. On the other hand, each of these components has spectral radius strictly less than 4, so one can conclude that the largest eigenvalue of $|M_v|$ is less than 4, and so is that of $M_v$. By Theorem \ref{thm 2-10}, $\mu_2(G)< 4$ and so $\mu_2(H)< 4$, as desired.
\end{proof}

\begin{theorem}\label{thm 3-3}
If $H$ is $L$-cospectral with $G = G(s, k)$, then they have the same degree sequence.
\end{theorem}

\begin{proof} By Lemma \ref{lem 3-2}, $\mu_2(H)<4$, and thus it follows from Theorem \ref{thm 2-6} that $d_2(H)\leq 3$. Since $H$ and $G$ are $L$-cospectral, by Theorem \ref{thm 2-1}, $H$ is also connected, and has the same order, size, and sum of the squares of its degrees as $G$. Let $n_i$ denote the number of vertices of degree $i$ in $H$, for $i=1, 2, \ldots, d_1(H)$. Then
\begin{equation}
\sum\limits_{i = 1}^{{d_1}(H)} {{n_i}} = n(G), \tag{1} \end{equation}
\begin{equation}
\hspace{3mm}\sum\limits_{i = 1}^{{d_1}(H)} {{in_i}} = 2m(G), \tag{2} \end{equation}

\begin{equation}
\hspace{20mm}\sum\limits_{i = 1}^{{d_1}(H)} {{i^2n_i}} ={ n'}_1+{4n'}_2+d^2_1(G) \tag{3},
\end{equation}
\noindent where $n'_1$ and $n'_2$ are the number of vertices of degree $1$ and $2$ in $G$, respectively.

It is clear that $n(G)=n$, $m(G)=n+s-1$, $n'_1=k$, $n'_2=n-(k+1)$ and $d_1(G)=2s+k$. By adding (1), (2), and (3) with coefficients $2, -3, 1$, respectively, we get:
\begin{equation}
\hspace{50mm}\sum\limits_{i = 1}^{{d_1}(H)} {(i^2-3i+2)n_i} = 4s^2+4sk-6s+k^2-3k+2 \tag{4}.
\end{equation}

By Lemma \ref{lem 3-1}, $2s+k+1\leq \mu_1(H)\leq 2s+k+2$. It follows from Theorem \ref{thm 2-4} that $d_1(H) +1\leq \mu_1(H)=\mu_1(G)\leq 2s+k+2$, which leads to $d_1(H)\leq 2s+k+1$. On the other hand, by Lemma \ref{lem 3-2} and Theorem \ref{thm 2-6} one can conclude that $2s+k+1\leq \mu_1(G)= \mu_1(H)\leq d_1(H) +d_2(H)\leq d_1(H)+3$, which leads to $d_1(H)\geq 2s+k-2$. Therefore, we have $2s+k-2\leq d_1(H)\leq 2s+k+1$. It follows from Theorem \ref{thm 2-8} that
\begin{equation}
6N_H(C_3)-\sum\limits_{i = 1}^{{n}} {{(d_i(H)-2)^3}} = 6s-(-k+8s^3+(k-2) [(k-2)^2+12s^2+6s(k-2)]) \tag{5}. \end{equation}

Our main proof will consider some cases as follows:\vspace{2mm}
\begin{enumerate}

\item $d_1(H)=2s+k-2$. We first assume that $n_{2s+k-2} = 1$. In this case, $2s+k-2=d_1(H)> 3 \geq d_2(H)$. From (4) and by a straightforward calculation, we get:
\begin{equation}
(2s+k-2)^2-3(2s+k-2)+2+2n_3=4s^2+4sk-6s+k^2-3k+2. \tag{6}
\end{equation}
\noindent from which we conclude that $n_3=4s+2k-5$. By Equations (2) and (3), it follows that $n_2=n-8s-5k+11$ and $n_1=4s+3k-7$. Furthermore, from (5) we deduce that $N_H(C_3)=(-k^2+6k)+(-4s^2+13s-9-4ks)$. Set $f(k)=-k^2+6k$ and $g(s, k)=-4s^2+13s-9-4ks$. So, $N_H(C_3)=f(k)+g(s, k)$. Obviously, for $k > 0$, $f(k)$ is non-negative if and only if $k \leq 6$. If $g(s, k)=0$, then $k=-\dfrac{4s^2-13s+9}{4s}$. Since $k>0$, $4s^2-13s+9<0$, and this holds if and only if $1 \leq s \leq \dfrac{9}{4}$. It is easy to see that if $s=2$, then $k=\dfrac{1}{8}$, and if $s=1 $, then $k=0$, both of which yield contradictions. Therefore, $g(s, k)$ has no roots for any natural numbers $s$ and $k$. Hence, $g(s, k)$ must always be negative. This means that for $k\geq 7$, we always have $N_H(C_3)=f(k)+g(s, k)<0$, again a contradiction. Now, if $k \in \left\{ {1, 2} \right\}$, then $N_H(C_3)=f(1)+g(s, 1)=-4s^2+9s-4$ and $N_H(C_3)=f(2)+g(s, 2)=-4s^2+5s-1$. If $s=1$, then $f(1)+g(s, 1)>0$ and $f(1)+g(s, 1)<0$, otherwise. Similarly,  If $s=1$, then $f(2)+g(s, 2)=0$ and $f(2)+g(s, 2)<0$, otherwise. Therefore, for $s\geq 2$ we get  $N_H(C_3)<0$, which is impossible. If $s=1$, then  $k=1, 2$; that is, $(s, k)=(1, 1)$ and $(s, k)=(1, 2)$ are contradict to  $2s+k-2> 3$. It is easy to check   that if $k \in \left\{ {3,4,5,6} \right\}$, then for any natural number $s$ we always have $N_H(C_3)<0$. Hence for any natural number $k$, $N_H(C_3)<0$ and this is obviously a contradiction.\vspace{2mm}

Next we assume that $n_{2s+k-2}\geq 2$. Then $2s+k-2=d_1(H)=d_2(H)\leq 3$, which implies that the pair $(k, s)$ equals $ (1, 1), (1, 2), (2, 1)$, or $(3, 1)$. So we need consider the following four subcases:

\begin{enumerate}
\item $(k,s)=(1,1)$. Therefore, $1=d_1(H)=d_2(H)$. This means that $H=K_2$, since $H$ is a connected graph. On the other hand, $(k,s)=(1,1)$ means that $n=n(H)\geq 4$, a contradiction.

\item $(k,s)=(1,2)$. Therefore, $3=d_1(H)=d_2(H)$. By (1), (3) and (4), $n_1=4$, $n_3=6$ and $n_2=n-10$. Now, by (5) we get $N_H(C_3)=-2$, a contradiction.

\item $(k,s)=(2,1)$.  Therefore, $2=d_1(H)=d_2(H)$. If all of the degrees are $2$, since $H$ is connected, it is a cycle of length at least $5$. On the other hand, in a $2$-regular connected graph, $\dfrac{2(n+s-1)} {n}=2$, so $s=1$, then $H=C_3$, a contradiction. Hence $H$ has some vertices of degree 1. Hence, all the vertices of $H$ have either degree 1 or 2, which means that $H$ is a path, and by (5) we get $2=0$, which is impossible.

\item $(k,s)=(3,1)$. Therefore, $3=d_1(H)=d_2(H)$. By (1), (3) and (4),  $n_1=n_3=6$ and $n_2=n-12$. Now, by (5) we have $N_H(C_3)=-3$, a contradiction.\vspace{2mm}
\end{enumerate}

\item  $n_{2s+k-1}= 1$. Then $2s+k-1=d_1(H)$. By an argument similar to that for (6), we get the following:
\begin{equation}
(2s+k-1)^2-3(2s+k-1)+2+2n_3=4s^2+4sk-6s+k^2-3k+2, \tag{7}
\end{equation}
Hence, $n_3=2s+k-2$. Combining Equations (2) and (3), we find that the roots are $n_1=2s+2k-3$ and $n_2=n-4s-3k+4$. Now, from (5) it is easy to see that $N_H(C_3)=\dfrac{-k^2+5k-6-4s^2-4ks+12s}{2}$. We claim that $N_H(C_3)<0$ or $(-k^2+5k-6)+(-4s^2-4ks+12s)<0$. Let $t(s, k)=-4s^2-4ks+12s$ and $l(k)=-k^2+5k-6$, so $N_H(C_3)=l(k)+t(s, k)$. It is easy to see that $l(k)$ is non-negative if $k = 2$ or $3$ and is negative otherwise.  If  $k\geq 3$, then $t(s, k)< 0$.  Therefore, for $k\geq 3$, $N_H(C_3)=l(k)+t(s, k)<0$, which is impossible. Consequently, there are two subcases to consider:
\begin{enumerate}
\item $k=1$. Then $N_H(C_3) = -2s^2+4s-1$.   For $s\geq 2$, $N_H(C_3)<0$, an impossibility.  If $s=1$, then $1=n_{2s+k-1}= n_2$ and also $N_H(C_3) = 1$. Therefore, we have only one vertex with degree 2. This means that  the  degree of other  vertices is 1 or $H=K_{1, 2}$, a contradiction. Note that, since $d_1(H)=2$,  $n_3=0$.

\item $k=2$. Then $N_H(C_3) = -2s^2+2s$.  Clearly, for all $s\geq 2$, $N_H(C_3)<0$, again an impossibility. If $s=1$, then $1=n_{2s+k-1}= n_3$. On the other hand, by (4) for $k=2$ and $s=1$, we get $n_3=3$,  a contradiction.\vspace{2mm}
\end{enumerate}

Now, we assume that $n_{2s+k-1}\geq 2$.  Then $2s+k-1=d_1(H)=d_2(H)\leq 3$ and so $(s,k)=(1, 1)$ or $(s,k)=(1, 2)$. We consider these two cases separately. Assume that $(s,k)=(1,1)$. Then $d_1(H)=d_2(H)=2$. If all of the degrees are $2$, then $H$ must be a cycle of length at least $4$, since it is connected. On the other hand, since $H$ is a $2$-regular connected graph, then $ \dfrac{2(n+s-1)}{n}=2$ so $s=1$ and $H=C_3$, a contradiction. Hence $H$ must have some vertices of degree 1, and hence must be a path. But then through (5) we get a contradiction. Finally, we assume that $(s,k)=(1,2)$, then $d_1(H)=d_2(H)=3$. Hence, by (4) $n_3=3$. Combining (2) and (3), we find that the roots are $n_1=3$ and $n_2=n-6$. Hence, $H$ has 3 vertices of degree 1. Let $v$ be such a vertex. It follows from Lemma \ref{lem 3-2} and  Theorem \ref{thm 2-10} that $4>\mu_2(H)\geq \mu_2(H-v)$. By Theorem \ref{thm 2-9}, each connected component of $H-v$ is either a path or an odd cycle. But this contradicts by the fact that $H-v$ has at least two vertices of degree 3.\vspace{2mm}

\item $d_1(H)=2s+k$. If $s=k=1$, then $d_1(H)=3$.  By (4), $n_3=1$. If $s \neq 1$ or $k \neq 1$, then $d_1(H) > 3$ and so $d_1(H) = 2s+k > 3\geq d_2(H)$. Hence one may deduce that for any natural numbers $s, k$ we always have $n_{2s+k}=1$. Now, by (4) one can deduce that
\begin{equation}
(2s+k)^2-3(2s+k)+2+2n_3=4s^2+4sk-6s+k^2-3k+2, \tag{8}
\end{equation}
From this it follows that $n_3=0$. Combining (2) and (3), we find that $n_1=k$ and $n_2=n-(k+1)$. Therefore the degrees of $H$ are the same as those of the graph $G$.\vspace{2mm}

\item $d_1(H)=2s+k+1$. Clearly, $n_{2s+k+1}=1$, since $d_1(H) = 2s+k +1 > 3\geq d_2(H)$. From (4) we deduce
\begin{equation}
(2s+k+1)^2-3(2s+k+1)+2+2n_3=4s^2+4sk-6s+k^2-3k+2, \tag{9}
\end{equation}
from which it follows that $n_3=-2s-k+1<0$, which is impossible.
\end{enumerate}
Hence the claim holds.
\end{proof}

Before proving our main result, we state some  essential lemmas and notations.

\begin{lemma}[\cite{Guo}]\label{lem 3-6}
Let $v$ be a vertex of a connected graph $G$ and suppose that $v_1,\cdots, v_s$
are pendant vertices of $G$ which are adjacent to $v$. Let $G^{*}$  be the graph obtained from $G$ by adding any $t$
$(1\leq t \leq \dfrac{s(s-1)}{2})$ edges among $v_1, v_2, \cdots , v_s$. Then we have $\mu_1(G)=\mu_1(G^{*})$.
\end{lemma}

\begin{lemma}[\cite{FY}]\label{lem 3-7}
No two non isomorphic starlike trees are $L$-cospectral.
\end{lemma}

Let $H$ be a path-friendship graph and let  $v $ be the maximum degree of $H$. Now, we remove an edge of any of  triangles, except in edges adjacent to $v$,  then we have a starlike tree, say, $S(H)$.  In the following,  $S(H)=(2s, t_1, \cdots, t_k)$ means that $S(H)-v=2sK_1 \cup P_{t_1}\cup \cdots \cup P_{t_k}$. Also, $G = G(s, t_1, t_2, \ldots, t_k)$ is a path-friendship graph having $s$ triangles and  $k$ paths with lengths $t_i$, $i = 1, 2, \ldots, k$. \vspace{2mm}

Note that in the proof of Lemma \ref{lem 3-7}, it have been shown that if $S_1=S(t_1, ..., t_k)$ and $S_2=S(l_1, ..., l_k)$ are two non-isomorphic starlike trees, then $\mu_1(S_1)\neq \mu_1(S_2)$, where $t_1\geq t_2\geq ... \geq t_k\geq 1$  and $l_1\geq l_2\geq ...\geq l_k\geq 1$.\vspace{2mm}

\begin{corollary}\label{cor 3-8}
Let $G=G(s, t_1, t_2, \ldots, t_k)$ and $H=H(s, l_1, l_2, \ldots, l_k)$ be two path-friendship graphs. If $S(G)=(2s, t_1, t_2, \ldots, t_k)$ and $S(H)=(2s, l_1, l_2, \ldots, l_k)$ are two non-isomorphic starlike trees,  then $\mu_1(S(G))\neq \mu_1(S(H))$.
\end{corollary}

\hspace{10mm}
\begin{figure}[h]
\centerline{\includegraphics[height=8cm]{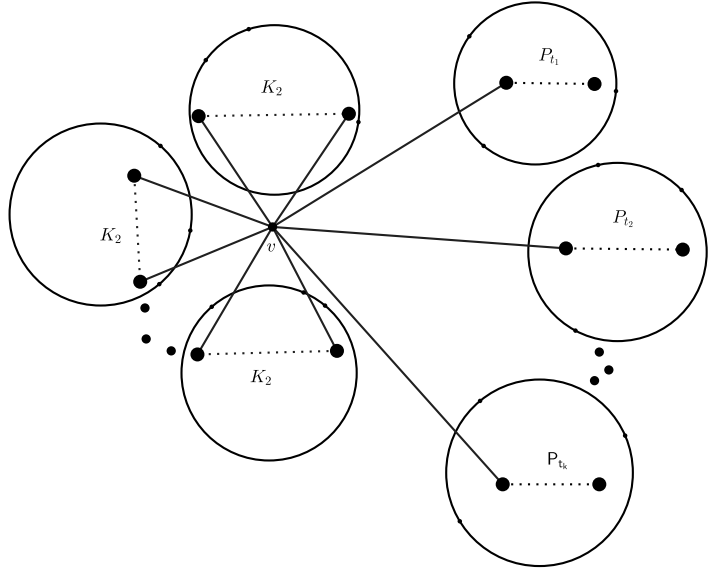}}
\begin{center}
{\caption{{The Path-friendship graph $G(s,k)$ and its connected components after removing the vertex $v$.}}
}\end{center}
\end{figure}

\begin{theorem}\label{thm 3-4}
If $H$ is $L$-cospectral with $G = G(s, k)$, then it is a path-friendship graph.
\end{theorem}

\begin{proof} Assume that $H$ is $L$-cospectral with $G = G(s, k)$. If $s = 0$, then $G$ is a starlike tree and  so  $H$ has the same graph structure.  Similarly, if $k = 0$, then $G$ is a friendship graph, and hence  $H$ has the same property. We now assume that $s, k > 0$. By Theorem \ref{thm 3-3}, $H$  has exactly one vertex of degree greater than $2$, say $\deg v = d > 2$. Consequently, $H - v$ has maximum degree at most 2. Furthermore, $H - v$ can have no cycles since if it did, then $H$, being connected, there would be another vertex of degree greater than 2. Consequently, $H - v$ must be a forest in which each component is a path. By the Corollary \ref{lem 2-10}, $H$ must have $s$ triangles. Furthermore, it has no other cycles. It follows that $H$ is a path-friendship graph with $s$ triangles and $k$ paths.
\end{proof}

Theorem \ref{thm 3-4} leads to our main result, that path-friendship graphs are DLS.

\begin{theorem}\label{thm 3-5}
All path-friendship graphs are DLS.
 \end{theorem}

\begin{proof} Let $G = G(s, t_1, t_2, \ldots, t_k)$ be the path-friendship graph having $s$ triangles and also $k$ paths with lengths $t_i$, $i = 1, 2, \ldots, k$. Further, assume that $H$ is a connected graph that is $L$-cospectral with $G$, but $H \ncong G$. By Theorem \ref{thm 3-4}, $H$ is a path-friendship graph, and has the same number of triangles as $G$, and also as many paths in its canonical form; that is, $H = H(s, l_1, l_2, \ldots, l_k)$. With the convention that the path-lengths are in non-decreasing order, we may assume that for some $i$, $l_i \neq t_i$. Now consider the starlike tree $S(G) = (2s, t_1, t_2, \ldots, t_k)$ with basic paths of the given lengths, and similarly $S(H) = (2s, l_1, l_2, \ldots, l_k)$. Then by Lemma \ref{lem 3-6}, $\mu_1(G) = \mu_1(S(G))$ and $\mu_1(H) =
\mu_1(S(H))$. However, we deduce from Corollary  \ref{cor 3-8}  that $\mu_1(S(G)) \neq \mu_1(S(H))$, and therefore $\mu_1(G) \neq \mu_1(H)$. But this contradicts by this hypothesis that $G$ and $H$ are $L$-cospectral.
\end{proof}

A consequence of this theorem is that the complements of path-friendship graphs are also DLS.

\vskip 3mm

\noindent\textbf{Acknowledgement}. The research of the second author is partially supported by the University of Kashan under grant no 890190/1.

\end{document}